\documentclass{amsart}
\usepackage{graphicx, amsmath, amsthm, amssymb, verbatim}
\usepackage[left=3cm,top=2cm,right=3cm]{geometry}
\usepackage[normalem]{ulem}
\newtheorem{thm}{Theorem}[section]

\newtheorem{prop}[thm]{Proposition}
\theoremstyle{definition}
\newtheorem{defn}[thm]{Definition}
\theoremstyle{remark}
\newtheorem{rmk}[thm]{Remark}
\numberwithin{equation}{section}
\newcommand{\norm}[1]{\left\Vert#1\right\Vert}
\newcommand{\abs}[1]{\left\vert#1\right\vert}

\newcommand{\Real}{\mathbb R}
\newcommand{\eps}{\varepsilon}

\begin{document}
\title{Incompressible Boussinesq Equations and Borderline Besov Spaces}
\author{Jacob Glenn-Levin}
\address{Department of Mathematics, 1 University Station C1200, Austin, TX 78712-0257}
\email{jglennlevin@math.utexas.edu}
\begin{abstract}
We prove local-in-time existence and uniqueness of an inviscid Boussinesq-type system. We assume the density equation contains nonzero diffusion and that our initial vorticity and density belong to a space of borderline Besov type.
\end{abstract}
\maketitle
\section{Introduction}\label{IntSec}
Consider the two dimensional Boussinesq system given by
\begin{center}
$(B_{\kappa, \nu}) \left\{\begin{array}{lc}\partial_t u + (u, \nabla) u - \nu\Delta u + \nabla P = \left(
                                                                               \begin{array}{c}
                                                                                 0\\
                                                                                  \rho \\
                                                                               \end{array}
                                                                             \right)\\
\partial_t \rho + (u, \nabla)\rho = \kappa \Delta \rho\\
\text{div } u = 0 \\
u(x,0) = u_0(x),\;\; \rho(x,0) = \rho_0(x).\\
\end{array}\right.$
\vspace{10pt}
\end{center}
Where $u=(u_1, u_2)$ is the velocity field, $\rho$ is the scalar density and $P$ is the pressure. This system is used to model, among other things, the influence of gravitational forces on the motion of a lighter or denser fluid. As such, it has relevance in the study of atmospheric and oceanographic dynamics and turbulence where rotation or stratification occurs (cf. ~\cite{MR1965452}). Of mathematical interest, note that for $\rho \equiv 0$ we recover the incompressible Navier-Stokes and Euler Equations (for $\nu > 0$ and $\nu \equiv 0$, respectively). Furthermore, the system $(B_{0,0})$ has vortex stretching similiar to that of the of 3D axisymmetric flow, thus providing a model under which to approach the formation of finite time singularities (see ~\cite{MR1867882} for discussion of this relationship). While there has been some numerical study of finite time singularity for $(B_{0,0})$, the results remain inconclusive ~\cite{MR1252833,MR1167779}.

From an existence and regularity perspective, much progress has been made recently. When both $\kappa$ and $\nu$ are strictly positive, ~\cite{MR565993,MR1013438} give global existence of smooth solutions using standard energy method arguments. In the case of $\kappa \equiv \nu \equiv 0$, local well-posedness as well as a blow-up criterion similar to well-known result of Beale, Kato and Majda ~\cite{MR763762} has been shown in a number of function spaces ~\cite{MR1711383,MR1475638,MR2645152}. For $\kappa \equiv 0$, $\nu > 0$, ~\cite{MR2239910} shows global existence of a solution to $(B_{0,\nu})$ for nondecaying initial data: $(u_0, \rho_0) \in \mathrm{L}^\infty (\Real^2) \times \dot{B}^0_{\infty, 1}(\Real^2)$.

In this paper, we study the case of $\nu \equiv 0, \kappa > 0$. Recall that in the celebrated paper ~\cite{MR0158189}, V. Yudovich proves the existence and uniqueness of a weak solution to the 2D incompressible Euler Equations,
\begin{center}
(E) $\left\{\begin{array}{lc}\partial_t u + (u, \nabla) u + \nabla P = 0\\
\text{div } u = 0 \\
u(x,0) = u_0(x),
\end{array}\right.$

\end{center}
for initial vorticity, $\omega_0 = \text{curl } u_0 \in \mathrm{L}^\infty(\Real^2)$. This result has been extended for initial vorticity in more general function spaces including, among others, Besov spaces. Recall that the (inhomogeneous) Besov Space, $B^s_{p,q}$, can be characterized as the set of all $f \in \mathcal{S}'$ such that
\begin{equation*}
\left(\sum_{j=-1}^\infty 2^{jqs}\norm{\Delta_j f}_p^q\right)^{\frac{1}{q}} < \infty,
\end{equation*}
where $s\in \Real$, $p,q\in [1,\infty]$ and $\Delta_j$ is the Littlewood-Paley operator defined below. In ~\cite{MR1717576}, M. Vishik proves the existence and uniqueness of a local-in-time solution to the incompressible 2D Euler Equations with initial data in a critical Besov-type Space (critical in the sense $s = n/p$ - see ~\cite{MR1851996} for a thorough discussion of critical, sub-critical and super-critical Besov Spaces). He introduces the $B_\Gamma$ spaces, based on the $B^0_{\infty, 1}$-norm, such that
\begin{equation*}
B_\Gamma = \left\{ f \in \mathcal{S}'(\Real^n) \vert \sum_{j=-1}^N \norm{\Delta_j f}_\infty = O(\Gamma(N))\right\},
\end{equation*}
and shows that for initial data in such a space, sufficient control on the growth of $\Gamma$ as $N\rightarrow \infty$ gives existence in a slightly weaker $B_\Gamma$-type space for positive time. More recently, R. Danchin and M. Paicu in ~\cite{MR2520505} prove a Yudovich-type result in $\Real^2$ for $(B_{\kappa, 0})$ under the assumption that $u_0 \in \mathrm{L}^2$, $\omega_0 \in \mathrm{L}^r \cap \mathrm{L}^\infty$ ($r\geq 2$) and $\rho_0 \in \mathrm{L}^2 \cap B^{-1}_{\infty, 1}$. In this paper, we seek to relax the boundedness constraints on the initial data and replace them with the more general membership in $B_\Gamma$ for suitable choice of $\Gamma$.

\begin{rmk}
Note that for $\Gamma$ such that $\Gamma(\alpha) =  C\alpha$  when  $\alpha \geq 1$, $\Gamma(\alpha) = 1$ otherwise, we have $\mathrm{L}^\infty \subset B_\Gamma$, thus the two spaces are comparable for such a choice of $\Gamma$.
\end{rmk}

The structure of the paper is as follows: In the remainder of this section we define more precisely some concepts we will need throughout the paper. In section \ref{APSec}, we prove a priori estimates for vorticity and density, generalizing the results found in ~\cite{MR1717576} and ~\cite{MR2520505} as well as a technical commutator estimate needed in the remainder of the paper. Finally, in sections \ref{USec} and \ref{ExSec} we prove the uniqueness and existence of these solutions, respectively.

In order to define the $B_\Gamma$ spaces discussed in this paper, we first need to define the Littlewood-Paley decomposition. Let $\Phi \in \mathcal{S}(\Real^n)$ be a function such that\\
supp $\hat{\Phi} \subset \{\xi \in \Real^n\;|\; \abs{\xi} \leq 1\}$ and $|\hat{\Phi}(\xi)| \geq C > 0$ on $\{\abs{\xi} \leq \frac{5}{6}\}$. Let $\varphi \in \mathcal{S}(\Real^n)$ be a radial function such that supp $\hat{\varphi} \subset \{\xi \in \Real^n \;|\; \frac{1}{2} \leq |\xi| \leq 2\}$, $|\hat{\varphi}(\xi)| \geq C > 0$ on $\{\frac{3}{5} \leq \abs{\xi} \leq \frac{5}{3}\}$. Set $\varphi_j(x) = 2^{jn}\varphi(2^j x)$ for $j \in \mathbb{Z}$ (i.e., $\hat{\varphi}_j(\xi) = \hat{\varphi}(2^{-j}\xi))$. Based on this choice of $\Phi$ and $\varphi$, the Littlewood-Paley operators are then given by:
\begin{defn}
 For $f \in \mathcal{S}'(\Real^n)$, we have
\begin{align*}
&\Delta_{-1} f = \hat{\Phi}\left(\frac{1}{i}\frac{\partial}{\partial x}\right) f = \mathcal{F}^{-1}(\hat{\Phi} \cdot \hat{f}) = \Phi \star f,\\
&\Delta_j f = \hat{\varphi}_j \left(\frac{1}{i}\frac{\partial}{\partial x}\right) f = \mathcal{F}^{-1}(\hat{\varphi} \cdot \hat{f}) = \varphi_j \star f \hspace{10pt}\text{for j} \geq 0,\\
&\Delta_j f = 0 \hspace{10pt}\text{for j} \leq -2\\
&S_j f = \sum_{k=-1}^j \Delta_k f.
\end{align*}
\end{defn}
\noindent From the above, we define the Littlewood-Paley decomposition of $f \in \mathcal{S}'$ as
\begin{equation*}
f = \sum_{j=-1}^\infty \Delta_j f.
\end{equation*}

\begin{defn} The space $B_\Gamma$ is the set of all $f\in \mathcal{S}'$ such that for any $N \geq -1$:
\begin{equation*}
\sum_{j=-1}^N \norm{\Delta_j f}_\infty \leq C \Gamma(N), \hspace{10pt}\text{where } \norm{f}_\Gamma = \sup_{N \geq -1} (\Gamma(N))^{-1} \sum_{j=-1}^N \norm{\Delta_j f}_\infty.
\end{equation*}
\end{defn}
For the purposes of this paper, let $\Gamma:\Real \rightarrow [1,\infty)$ satisfy the following conditions:
\renewcommand{\theenumi}{\roman{enumi}}
\begin{enumerate}
  \item  $\Gamma(\alpha) = 1$ for $\alpha \in (-\infty, -1], \lim_{\alpha \rightarrow \infty} \Gamma(\alpha) = \infty$
  \item  There is a constant $C > 0$ such that $C^{-1}\Gamma(\beta) \leq \Gamma(\alpha) \leq C\Gamma(\beta)$ for  $\alpha,\beta \in [-1, \infty), \\|\alpha - \beta| \leq 1.$
  \item There is a constant $C > 0$ such that for $\alpha \in [-1,\infty),$
  \[ C2^{-\alpha}\Gamma(\alpha) \geq \int_\alpha^\infty 2^{-\xi}\Gamma(\xi)\text{d}\xi. \]
\end{enumerate}
Define $\Gamma_1(\alpha) = (\alpha +2)\Gamma(\alpha)$ for $\alpha \geq -1$, $\Gamma_1(\alpha) = 1$ otherwise and assume:
\begin{enumerate}
  \setcounter{enumi}{3}
    \item $\Gamma_1$ satisfies (iii),
    \item $\Gamma_1$ is convex,
    \item $\int_1^\infty (\Gamma_1(\alpha))^{-1}\text{d}\alpha = \infty$.
\end{enumerate}

\section{A Priori and Commutator Estimates}\label{APSec}
Let $\mathcal{K}$ be the Biot-Savart kernel. Denote by $X_u(x,t;\tau) := X(x,t;\tau)$ the flow mapping generated by $u$. $X$ solves 
\[
\frac{\partial}{\partial t}X(x,t;\tau) = u(X(x,t;\tau),t)\; ;\;\;\; X(x,0;\tau)=x(\tau)
\] 
and for any $0\leq \tau \leq t \leq T$, $X(\cdot, t;\tau)$ is a volume preserving homeomorphism $\Real^2 \stackrel{\text{onto}}{\longrightarrow} \Real^2$. The goal for this section is to prove the following a priori bound for $\omega$:

\begin{thm}\label{apriorithm}
Let $1 < p_0 < 2 < p_1 < \infty$. Let $f \in B_\Gamma \cap \mathrm{L}^{p_0} \cap \mathrm{L}^{p_1}$, and let $g \in \mathrm{W}^{1,p_0} \cap \mathrm{W}^{1,p_1}$ such that $\nabla g \in B_\Gamma$. Let  $u = \mathcal{K} * \omega$, $\omega_0 = f$ and $\rho_0 = g$. Finally, let $(u,\rho)$ solve ($B_{\kappa, 0}$). then we have the following a priori estimates:
\renewcommand{\labelenumi}{\arabic{enumi}.}
\begin{enumerate}
  \item There exists $T > 0$ (depending on $\Gamma$, $f$ and $g$) such that $\omega, \nabla \rho \in \mathrm{L}^\infty([0,T]; \mathrm{L}^{p_0} \cap \mathrm{L}^{p_1}), \nabla \rho \in \mathrm{L}^\infty([0,T]; B_{\Gamma}),$ and $\omega \in \mathrm{L}^\infty([0,T]; B_{\Gamma_1})$
when \begin{equation}\label{Gamma1}
      (\alpha +2)\Gamma'(\alpha) \leq C \text{ for a.e. } \alpha \in [-1,\infty).
     \end{equation}

  \item Furthermore, $\omega, \nabla \rho \in \mathrm{L}^\infty_{\text{loc}}([0,\infty); \mathrm{L}^{p_0} \cap \mathrm{L}^{p_1}), \nabla \rho \in \mathrm{L}^\infty_{\text{loc}}([0,\infty); B_{\Gamma})$,\\
   and $\omega \in \mathrm{L}^\infty_{\text{loc}}([0,\infty); B_{\Gamma_1})$ under the stronger assumption that \begin{equation}\label{Gamma2}
                     \Gamma'(\alpha)\Gamma_1(\alpha)\leq C \text{ for a.e. } \alpha \in [-1,\infty).
                    \end{equation}
\end{enumerate}
\end{thm}
\begin{rmk}
Note that while we do not explicitly use assumptions (\ref{Gamma1}) and (\ref{Gamma2}) in the main body of this paper, they are an essential piece of Proposition \ref{vprop1}.
\end{rmk}

Observe that if we apply the curl operator to the equation satisfied by the velocity field $u$ then curl $\nabla P$ = 0, and we have the vorticity equation $\partial_t \omega + (u, \nabla) \omega = \partial_1 \rho$. Integrating along characteristic curves, we have that for flow map $X(x,t;\tau) = X_u(x,t;\tau)$,
\begin{equation*}
\omega(x(t),t) = \omega_0(X^{-1}(t;0)x(t)) + \int_0^t \partial_1 \rho(X^{-1}(t;\tau)x(t),\tau)\text{d}\tau.
\end{equation*}
For any $p \in [1, \infty)$, this implies that
\begin{equation}\label{eq1.0k}
\norm{\omega(t)}_p \leq \norm{\omega_0}_p + \int_0^t \norm{\nabla \rho}_p\text{d}\tau,
\end{equation}
since $X$ is a volume-preserving homeomorphism.
In regards to the $B_{\Gamma_1}$ norm, an initial estimate gives us:
\begin{equation}\label{eq1.0g}
\norm{\omega(t)}_{\Gamma_1} \leq \norm{\omega_0(X^{-1}(t))}_{\Gamma_1} + \int_0^t \norm{\nabla \rho(X^{-1}(t;\tau),\tau)}_{\Gamma_1}\text{d}\tau.
\end{equation}
With (\ref{eq1.0g}) in mind, we must address the following two concerns: First, does the gradient of density remain in the initial $B_\Gamma$ space for positive time? Assuming the first question, how does the inverse flow map $X^{-1}$ act on the $B_\Gamma$ spaces? In order to address this second question, we make use of a key proposition of Vishik in ~\cite{MR1717576}:
\begin{prop}[Vishik]\label{vprop1}
Let $\omega_0 \in B_\Gamma \cap \mathrm{L}^{p_0} \cap \mathrm{L}^{p_1}$, and let $\rho_0 \in \mathrm{W}^{1,p_0} \cap \mathrm{W}^{1,p_1}$ such that $\nabla \rho_0 \in B_\Gamma$.  Let ($u,\rho$) be a regular solution to ($B_{\kappa, 0}$) such that $u\in \mathcal{K} * C([0,T]; B_{\Gamma_1} \cap \mathrm{L}^{p_0} \cap \mathrm{L}^{p_1})$, where $T$ is defined for each case below. Let $X_u(t;\tau)$ be the flow map given by $u$, and let $f \in B_\Gamma \cap \mathrm{L}^{p_0} \cap \mathrm{L}^{p_1}$ be an arbitrary function. Then we have the following estimates:

\renewcommand{\labelenumi}{\arabic{enumi}.}
\begin{enumerate}
\item If (\ref{Gamma1}) holds, then there exists a $T > 0$ and a $C > 0$ (both depending on $\Gamma$ and the initial data), such that for $0\leq \tau \leq t \leq T$,
\begin{equation*}
\norm{f \circ X^{-1}_u(t;\tau)}_{\Gamma_1} \leq C \norm{f}_\Gamma.
\end{equation*}
\item Under assumption (\ref{Gamma2}), let $T> 0$ be arbitrary. Then there exists $\lambda(\cdot) \in \mathrm{L}^\infty_{\text{loc}}(0,\infty)$ (depending on $\Gamma$ and the initial data), such that:
\begin{equation*}
\norm{f \circ X^{-1}_u(t;\tau)}_{\Gamma_1} \leq C\norm{f}_\Gamma \lambda(t)
\end{equation*}
for all $0 \leq \tau \leq t \leq T$.
\end{enumerate}
\end{prop}
\begin{rmk}
For ease of exposition, we fix $T >0$ for the remainder of this section. In the case of assumption (\ref{Gamma1}), this $T$ depends on our choice of $\Gamma$, while under assumption (\ref{Gamma2}), this choice of $T$ is arbitrary.
\end{rmk}

The second tool we need to prove Theorem \ref{apriorithm} is the following a priori control on the density $\rho$:
\begin{thm}\label{densitythm}
Assume $\omega_0$ and $\rho_0$ are defined as above. Set $\alpha = \left(\frac{1+\kappa t}{\kappa}\right)$. Then there is a constant $C > 0$ (depending only on  $\Gamma$, $f$, $g$ and $T$) such that for all $t >0$,  $\int_0^t \norm{\nabla \rho(\tau)}_{\mathrm{L}^{p_0} \cap B_\Gamma} \text{d}\tau \leq \Upsilon(t)$, where
\begin{align*} \Upsilon(t) = \left[\alpha^2\left(\norm{\rho_0}_{B^{-1}_{p_0,1}\cap B^{-1}_{\infty, 1}}^2 + \alpha t \norm{\rho_0}_2^2 \norm{\omega_0}_{\mathrm{L}^{p_0}\cap B_\Gamma}^2\right)\right]^{\frac{1}{2}} \exp \left[C\alpha^3 \norm{\rho_0}_2^2 t\right].
\end{align*}
\end{thm}
\noindent Using Proposition \ref{vprop1} together with Theorem \ref{densitythm}, we conclude that for $t \in [0,T]$, the terms on the right hand side of (\ref{eq1.0g}) are bounded by a constant multiple of $\norm{\omega_0}_\Gamma$ and $\int_0^t \norm{\nabla \rho(\tau)}_\Gamma \text{d}\tau$, respectively, hence proving Theorem \ref{apriorithm}.

It remains to prove Theorem \ref{densitythm}. First, note that $\rho$ solves
\begin{equation*}
\partial_t \rho - \kappa \Delta \rho = h,
\end{equation*}
where $h = -(u, \nabla) \rho$. Written in this form, it becomes evident that we may treat the right hand side as the nonhomogeneous part of a heat equation and make use of the smoothing properties of the Laplacian. Let $\{e^{\nu \Delta}\}_{\nu > 0}$ stand for the heat semi-group. Applying the $\Delta_j$ operator (for $j \geq 0$) to the above equation, we have
\begin{equation}\label{eq1.0o}
\partial_t \Delta_j \rho - \kappa \Delta \Delta_j \rho = \Delta_j h,
\end{equation}
which implies that
\begin{equation}\label{eq1.0l}
\Delta_j \rho(t) = e^{\kappa t \Delta}\Delta_j\rho_0 + \int_0^t e^{\kappa(t-\tau)\Delta}\Delta_j h(\tau)\text{d}\tau.
\end{equation}
As in ~\cite{MR2520505}, we make use of the following result of Chemin (see ~\cite{MR1753481} for proof):
\begin{prop}[Chemin]\label{Cheminprop} There exists positive constants $(C,c)$ such that for any $1 \leq p \leq \infty$ and $\nu > 0$,
\[\norm{e^{\nu \Delta} \Delta_j h}_p \leq Ce^{-c\nu 2^{2j}}\norm{\Delta_j h}_p.\]
\end{prop}

Chemin's proposition applied to (\ref{eq1.0l}) gives us
\begin{equation}\label{eq1.0m}\norm{\Delta_j \rho(t)}_p \leq C\left( e^{-C\kappa 2^{2j} t}\norm{\Delta_j\rho_0}_p + \int_0^t e^{-C\kappa 2^{2j} (t-\tau)}\norm{\Delta_j h(\tau)}_p\text{d}\tau\right),\end{equation}
and integrating both sides in $t$ leads to the following inequality for $j\geq 0$:
\begin{equation}\label{eq1.0n} \kappa 2^{j}\int_0^t\norm{\Delta_j \rho(\tau)}_p\text{d}\tau \leq C 2^{-j}\left(\norm{\Delta_j \rho_0}_p + \int_0^t \norm{\Delta_j h(\tau)}_p \text{d}\tau\right).\end{equation}
Let $1 \leq p \leq \infty$. From an application of the maximum principle to (\ref{eq1.0o}), we have
\begin{align*}
&\norm{\Delta_{-1}\rho(t)}_p \leq \norm{\Delta_{-1} \rho_0}_p + \int_0^t \norm{\Delta_{-1}h(\tau)}_p\text{d}\tau\\
&\Rightarrow \int_0^t \norm{\Delta_{-1} \rho(\tau)}_p\text{d}\tau \leq C t\left(\norm{\Delta_{-1} \rho_0}_p + \int_0^t \norm{\Delta_{-1} h(\tau)}_p \text{d}\tau\right).
\end{align*}
Summing $j \geq -1$, we conclude that for any $1 \leq p \leq \infty$,
\begin{equation}\label{eq1.0a}
\kappa \int_0^t \norm{\rho(\tau)}_{B^1_{p, 1}}\text{d}\tau \leq C(1+\kappa t)\left(\norm{\rho_0}_{B^{-1}_{p, 1}} + \int_0^t \norm{(u,\nabla)\rho}_{B^{-1}_{p, 1}} \text{d}\tau\right).
\end{equation}
Observe that for $p = \infty$, the norm on the left hand side of (\ref{eq1.0a}) is equivalent to the $B^0_{\infty,1}$-norm of $\nabla \rho$, which is itself an upper bound for $\norm{\nabla \rho}_\Gamma$. Similarly, for $p= p_0$, the $B^1_{p_0,1}$-norm of $\rho$ on the left hand side is a bound for $\norm{\nabla \rho}_{p_0}$. Therefore, in order to prove the a priori bound on the $B_\Gamma \cap \mathrm{L}^{p_0}$-norm of $\nabla \rho$, it suffices to control the right hand side of (\ref{eq1.0a}) by suitable bounds and then utilize a Gronwall-type estimate.

\begin{rmk}
For the space $B^{-1}_{p_0, 1}$, we use the embedding $\mathrm{W}^{1,p_0} \hookrightarrow B^1_{p_0,\infty} \hookrightarrow B^{-1}_{p_0,1}$ to conclude that $\rho_0 \in B^{-1}_{p_0, 1}$. For details, see ~\cite{MR0461123}.
\end{rmk}

Since the cases $p = \infty$ and $p = p_0$ are nearly identical, we will address only the case $p = \infty$. Use Bony's paraproduct decomposition to write
\[(u,\nabla)\rho = R(u, \nabla \rho) + \sum_{m=1}^2 T_{\partial_m \rho} u_m + T_{u_m}\partial_m \rho,\]
where $R(f,g) = \sum_{|j-k| \leq 1} \Delta_j f \Delta_k g,$ and $T_f g = \sum_{j=0}^\infty S_{j-2} f \Delta_j g.$
Since div $u = 0$, we have $R(u, \nabla \rho) = \text{div} R(u, \rho)$. Because div maps $B^0_{\infty, 1} \rightarrow B^{-1}_{\infty, 1}$, it suffices to bound $\norm{R(u,\rho)}_{B^0_{\infty, 1}}$. To do so, we write
\begin{align*}
\norm{R(u,\rho)}_{B^0_{\infty, 1}} &= \sum_{j=-1}^\infty \norm{\Delta_j \sum_{|k-l|\leq 1} \Delta_k u \Delta_l \rho}_\infty\\
&\leq \sum_{j=-1}^\infty \sum_{\stackrel{|k-l|\leq 1}{|j-k| \leq M_0}} \norm{\Delta_j \Delta_k u \Delta_l \rho}_\infty\\
&\leq C\sum_{j=-1}^\infty \sum_{|j-l| \leq M_0} \norm{\Delta_j u}_\infty \norm{\Delta_l \rho}_\infty\\
&\leq C\norm{\rho}_{B^0_{\infty,\infty}} \norm{u}_{B^0_{\infty, 1}},
\end{align*}
where $\norm{\rho}_{B^0_{\infty,\infty}} = \sup_{j\geq -1} \norm{\Delta_j \rho}_\infty$. Since $\norm{u}_{B^0_{\infty, 1}}$ is bounded by a constant multiple of $\norm{\omega}_{p_0} + \norm{\omega}_{\Gamma_1}$, we conclude that
\begin{equation}\label{eq1.0b}
\norm{R(u,\nabla \rho)}_{B^{-1}_{\infty, 1}} \leq C \norm{\rho}_{B^0_{\infty,\infty}}(\norm{\omega}_{p_0} + \norm{\omega}_{\Gamma_1}).
\end{equation}

For the second term in the paraproduct decomposition, we use the following estimate:
\begin{align*}
\norm{T_{\partial_m \rho} u_m}_{B^{-1}_{\infty, 1}} &= \sum_{j=-1}^\infty 2^{-j}\norm{\Delta_j(\sum_{k=0}^\infty S_{k-2} \partial_m \rho \Delta_k u_m)}_\infty\\
&\leq C \sum_{j=-1}^\infty \sum_{|j-k|\leq M_0} 2^{-j}\norm{\Delta_j \partial_m \rho}_\infty \norm{\Delta_k u_m}_\infty\\
&\leq C \norm{\rho}_{B^0_{\infty,\infty}} \sum_{k=-1}^\infty \norm{\Delta_k u}_\infty\\
&\leq C\norm{\rho}_{B^0_{\infty,\infty}} \norm{u}_{B^0_{\infty, 1}}.
\end{align*}
The above is true regardless of our choice of $m$, and a near identical argument shows the same expression bounds the $B^{-1}_{\infty, 1}$-norm of $T_{u_m} \partial_m \rho$. Thus
\begin{equation}\label{eq1.0c}
\norm{\sum_{m=1}^2 T_{\partial_m \rho} u_m + T_{u_m}\partial_m \rho}_{B^{-1}_{\infty, 1}} \leq  C \norm{\rho}_{B^0_{\infty,\infty}}(\norm{\omega}_{p_0} + \norm{\omega}_{\Gamma_1}).
\end{equation}
By H\"{o}lder's inequality, we can write
\begin{align}\label{eq1.0q}
\int_0^t \norm{\rho}_{B^0_{\infty,\infty}}(\norm{\omega}_{p_0} + \norm{\omega}_{\Gamma_1})\text{d}\tau &\leq\\
&\hspace{-60pt}\left(\int_0^t \norm{\rho}_{B^0_{\infty,\infty}}^2\text{d}\tau\right)^{\frac{1}{2}}\left(\int_0^t (\norm{\omega}_{p_0} + \norm{\omega}_{\Gamma_1})^2\text{d}\tau\right)^{\frac{1}{2}},
\end{align}
and bound each integral individually. To handle the first integral, observe that if we take the $\mathrm{L}^2$ inner product of $\rho$ with the equation satisfied by $\rho$, we have
\[\langle\rho, \partial_t \rho\rangle + \langle\rho, (u,\nabla)\rho\rangle + \langle\rho, \kappa \Delta \rho\rangle= 0,\]
and following an integration by parts in the space variable and a time integration over [0,t], 
\begin{equation}\label{eq1.0d}
\norm{\rho(t)}_2^2 + 2 \kappa \int_0^t \norm{\nabla \rho(\tau)}_2^2 \text{d}\tau = \norm{\rho_0}_2^2
\end{equation} for all $t \in \Real_+$. By the definition of $\norm{\rho}_{B^0_{\infty,\infty}}$ and Bernstein's inequality, we have
\begin{align}\label{eq1.0p}
\norm{\rho}_{B^0_{\infty,\infty}} &= \sup_{j \geq -1} \norm{\Delta_j \rho}_\infty\\
&\nonumber \leq \sup_{j \geq -1} 2^j \norm{\Delta_j \rho}_2\\
&\nonumber \leq C \norm{\Delta_{-1} \rho}_2 + \sup_{j \geq 0} 2^j \norm{\Delta_j \rho}_2\\
&\nonumber \leq C (\norm{\rho}_2 + \norm{\nabla \rho}_2).
\end{align}
Squaring both sides and integrating over time, we have by (\ref{eq1.0d})
\begin{align*}
\int_0^t \norm{\rho}_{B^0_{\infty,\infty}}^2\text{d}\tau &\leq C \int_0^t \norm{\rho}_2^2 \text{d}\tau + \int_0^t \norm{\nabla \rho}_2^2\text{d}\tau\\
&\leq C \norm{\rho_0}_2^2 t + \frac{1}{\kappa}\norm{\rho_0}_2^2\\
&\leq C \alpha\norm{\rho_0}_2^2.
\end{align*}
Combining the above with the bound given by (\ref{eq1.0a}), we conclude that
\begin{align}\label{eq1.0e}
\nonumber\int_0^t \norm{\nabla \rho(\tau)}_\Gamma\text{d}\tau &\leq C\alpha\left[\norm{\rho_0}_{B^{-1}_{\infty, 1}} + C\int_0^t \norm{\rho}_{B^0_{\infty,\infty}}(\norm{\omega}_{p_0} + \norm{\omega}_{\Gamma_1})\text{d}\tau\right]\\
& \leq C\alpha\left[\norm{\rho_0}_{B^{-1}_{\infty, 1}} + C \alpha^{\frac{1}{2}}\norm{\rho_0}_2 \left(\int_0^t (\norm{\omega}_{p_0} + \norm{\omega}_{\Gamma_1})^2\text{d}\tau\right)^{\frac{1}{2}}\right]
\end{align}
To achieve the desired a priori bound on the gradient of the density, it then suffices to control the $\mathrm{L}^2$ (in time) integral of the $\mathrm{L}^{p_0}$ and $B_{\Gamma_1}$ norms (in space) of vorticity.
By (\ref{eq1.0k}) and Proposition \ref{vprop1}, we have for all $t \in [0,T]$:
\begin{align}\label{eq1.0h}
&\norm{\omega(t)}_{p_0} \leq \norm{\omega_0}_{p_0} + \int_0^t \norm{\nabla \rho}_{p_0}\text{d}\tau,\\
&\norm{\omega(t)}_{\Gamma_1} \leq C(\norm{\omega_0}_\Gamma + \int_0^t \norm{\nabla \rho(\tau)}_\Gamma\text{d}\tau).
\end{align}
Define $\Theta(t) = \int_0^t \norm{\nabla \rho(\tau)}_{B_\Gamma \cap \mathrm{L}^{p_0}}\text{d}\tau$, where $\norm{\cdot}_{B_\Gamma \cap \mathrm{L}^{p_0}} = \max \left\{\norm{\cdot}_{p_0}, \norm{\cdot}_\Gamma \right\}$. Then inserting the above into (\ref{eq1.0e}) gives
\begin{align}\label{eq1.0i}
\nonumber \int_0^t \norm{\nabla \rho(\tau)}_\Gamma\text{d}\tau &\leq C\alpha\left[\norm{\rho_0}_{B^{-1}_{\infty, 1}}\right. + \\
& + C \left.\alpha^{\frac{1}{2}}\norm{\rho_0}_2 \left(t \norm{\omega_0}_{\mathrm{L}^{p_0}\cap B_\Gamma}^2
+ C \int_0^t \Theta^2(\tau) \text{d}\tau\right)^{\frac{1}{2}}\right].
\end{align}
\begin{rmk}The argument with respect to $B^0_{p_0,1}$ yields an identical estimate, with $\norm{\nabla \rho}_{p_0}$ and $\norm{\rho_0}_{B^{-1}_{p_0, 1}}$ replacing the first two norms in (\ref{eq1.0i}), respectively.\end{rmk}
\noindent Combining (\ref{eq1.0i}) and the equivalent $B^0_{p_0,1}$ estimate and squaring both sides, we have
\begin{align*}
\Theta^2(t) &\leq C\alpha^2 \left[\norm{\rho_0}_{B^{-1}_{p_0,1}\cap B^{-1}_{\infty, 1}}^2 + C \alpha\norm{\rho_0}_2^2 \left(t \norm{\omega_0}_{\mathrm{L}^{p_0}\cap B_\Gamma}^2
+ C \int_0^t \Theta^2(\tau) \text{d}\tau\right)\right]\\
&\leq C \alpha^2\left[\norm{\rho_0}_{B^{-1}_{p_0,1}\cap B^{-1}_{\infty, 1}}^2 + \alpha t \norm{\rho_0}_2^2 \norm{\omega_0}_{\mathrm{L}^{p_0}\cap B_\Gamma}^2\right]\\
& \hspace{60pt}+ C\alpha^3 \norm{\rho_0}_2^2 \int_0^t \Theta^2(\tau)\text{d}\tau.
\end{align*}
An application of Gronwall's inequality to $\Theta^2(t)$ gives
\begin{equation*}
\Theta^2(t) \leq C \alpha^2\left(\norm{\rho_0}_{B^{-1}_{p_0,1}\cap B^{-1}_{\infty, 1}}^2 + \alpha t \norm{\rho_0}_2^2 \norm{\omega_0}_{\mathrm{L}^{p_0}\cap B_\Gamma}^2\right) \exp \left[C\alpha^3 \norm{\rho_0}_2^2 t\right]
\end{equation*}
and taking a square root of both sides yields the desired bound.

Finally, in order to apply the a priori estimate to the proofs of uniqueness and existence, we first must understand what happens when the $\Delta_j$ operator is applied to the nonlinear term $(u, \nabla)\rho$. We follow the general approach introduced in ~\cite{MR1288809} and write
\[ R_j(u,\rho) = \Delta_j(u, \nabla)\rho - (S_{j-2}u, \nabla)\Delta_j \rho.\]
Let $M_0$ be the constant such that $\Delta_j \Delta_k f = 0$ if $|j-k| > M_0$. Note that $M_0$ depends strictly on our choice of $\varphi$ and $\Phi$ defining the $\Delta_j$ operators.
\begin{thm}\label{Comthm}
For $u,\rho$ defined above:
\begin{align}\label{eq2.1}
\norm{R_j(u,\rho)}_\infty \leq C \sum_{|j-l| \leq M_0} &\left\{ \norm{S_{l-2} \rho}_\infty \norm{\Delta_l \nabla u}_\infty +\norm{S_{l-2} \nabla u}_\infty \norm{\Delta_l \rho}_\infty \right\}\\
&\nonumber+ C 2^j \sideset{}{'}\sum_{\stackrel{l \geq j- M_0}{|l-m|\leq 1}} 2^{-l} \norm{\Delta_l \nabla u}_\infty \norm{\Delta_m \rho}_\infty
\end{align}
where the $\sideset{}{'}\sum$ implies that for $l = -1$, the factor $\norm{\Delta_l \nabla u}_\infty$ should be replaced by $\norm{\Delta_{-1} u}_\infty$.
\end{thm}
\begin{proof}
 The proof roughly follows that of Theorem 6.1 in ~\cite{MR1717576} and is therefore omitted.
\end{proof}

\section{Uniqueness of the flow}\label{USec}
Let $\Pi: \Real \rightarrow [1,\infty)$ be a function such that (i)-(iii) of Section (\ref{IntSec}) are satisfied. In addition, assume the following holds for $\Pi$:
\begin{align}
\label{2assump1}&\int_1^\infty [\Pi(\xi)]^{-1} \text{d}\xi = \infty\\
\label{2assump2}& \Pi(\xi)2^{-\xi} \text{ is nonincreasing for } \alpha \geq C \text{ , } \lim_{\xi\rightarrow \infty} \Pi(\xi)2^{-\xi} = 0.
\end{align}
\begin{rmk} We use $\Pi$ in place of $\Gamma$ and $\Gamma_1$ in this section since the uniqueness result utilizes weaker assumptions on $\Pi$ than those needed in Section \ref{APSec}.\end{rmk}
\begin{thm}\label{uthm}For $t \in [0,T]$, let $(u_1,\rho_1), (u_2,\rho_2)$ be two solutions to $(B_{\kappa,0})$, and let \\
$\omega_{1,2} = \text{curl } u_{1,2}$. We assume that for $ 1 < p_0 <2 $:
\begin{align}
\label{eq3.0.1}&\omega_{1,2}, \nabla \rho_{1,2} \in \mathrm{L}^\infty ([0,T]; \mathrm{L}^{p_0})\text{, } \norm{\omega_{1,2}(\cdot)}_{\Pi}, \norm{\nabla \rho_{1,2}(\cdot)}_\Pi \in \mathrm{L}^\infty([0,T]),\\
\label{eq3.0.2}&u_{1,2} = \mathcal{K} * \omega_{1,2},\\
\label{eq3.0.3}& \text{div }u_{1,2} = 0\\
\label{eq3.0.4}&\omega_{1,2}(\cdot,0)= f(\cdot), \rho_{1,2}(\cdot,0)=g(\cdot)\text{; } f,g\in B_\Pi \cap \mathrm{L}^{p_0}.
\end{align}
Then $(u_1,\rho_1)=(u_2,\rho_2)$ for $t\in [0,T]$.\end{thm}

\begin{proof}Define $v = u_1-u_2$, $\omega = \omega_1-\omega_2$, $\rho = \rho_1-\rho_2$ and $P = P_1-P_2$. We then have (for $\dot{f} = \frac{\partial}{\partial t} f$):
\begin{equation}\label{eq3diff}
\left\{\begin{array}{l}
\dot{v} = - (u_1,\nabla)v - (v,\nabla)u_2 - \nabla P + \left(
                                                         \begin{array}{c}
                                                           0 \\
                                                           \rho \\
                                                         \end{array}
                                                       \right)\\

\dot{\rho} - \kappa \Delta \rho=  -(u_1,\nabla)\rho + (v, \nabla)\rho_2\\
\text{div } v = 0\\
v|_{t=0} = \rho|_{t=0} = 0.
\end{array}\right.\end{equation}
We handle the equation for $\dot{\rho}$ first. Applying the $\Delta_j$ operator, we have
\[\Delta_j \dot{\rho} - \kappa \Delta_j \Delta \rho= -(S_{j-2} u_1, \nabla) \Delta_j \rho + R_j(u_1, \rho) + (S_{j-2} v, \nabla)\Delta_j \rho_2 + R_j(v,\rho_2).\]
Using (\ref{eq1.0m}) and the fact that $\rho(0) = 0$ gives
\begin{align}\label{eq3.1}
\norm{\Delta_j \rho(t)}_\infty &\leq \int_0^t e^{-C\kappa 2^{2j}(t-\tau)}(\norm{R_j(u_1, \rho)}_\infty \\ &\hspace{20pt}+ \nonumber \norm{(S_{j-2} u_1, \nabla) \Delta_j \rho}_\infty + \norm{R_j(v, \rho_2)}_\infty\\
&\hspace{20pt} + \nonumber \norm{(S_{j-2}v,\nabla)\Delta_j \rho_2}_\infty)\text{d}\tau.
\end{align}
It therefore suffices to bound the four terms on the right hand side. By Theorem \ref{Comthm}, we have for fixed $j$
\begin{align}
\nonumber\norm{R_j(u_1,\rho)}_\infty \leq C \sum_{|j-l| \leq M_0} &\left\{ \norm{S_{l-2} \rho}_\infty \norm{\Delta_l \nabla u_1}_\infty +\norm{S_{l-2} \nabla u_1}_\infty \norm{\Delta_l \rho}_\infty \right\}\\
&+\; C 2^j \sum_{l \geq j- M_0}\sideset{}{'}\sum_{|l-m|\leq 1} 2^{-l} \norm{\Delta_l \nabla u_1}_\infty \norm{\Delta_m \rho}_\infty.
\end{align}
Let $A_1$, $A_2$ be the sum of the first and second lines above from $j=-1$ to $N$, respectively. (We will determine $N \geq 1$ later.) Then we have
\begin{align*}
A_1 &= C\sum_{j=-1}^N \sum_{|j-l| \leq M_0} \left\{ \norm{S_{l-2} \rho}_\infty \norm{\Delta_l \nabla u_1}_\infty +\norm{S_{l-2} \nabla u_1}_\infty \norm{\Delta_l \rho}_\infty \right\}\\
&\leq C\left(\sup_{-1\leq l \leq N+M_0}\norm{S_{l-2} \rho}_\infty\right)\sum_{l=-1}^{N+M_0} \norm{\Delta_l \nabla u_1}_\infty\\
&\hspace{30pt} + C\left(\sup_{-1\leq l \leq N+M_0}\norm{S_{l-2} \nabla u_1}_\infty\right)\sum_{l=-1}^{N+M_0} \norm{\Delta_l \rho}_\infty\\
&\leq C \left(\sum_{l=-1}^{N+M_0} \norm{\Delta_l \nabla u_1}_\infty\right)\left(\sum_{l=-1}^{N+M_0} \norm{\Delta_l \rho}_\infty\right).
\end{align*}
Using  Bernstein's inequality, $\omega_1 \in B_\Pi$ and the fact that $\nabla u \mapsto \omega$ is a bounded operator on $\mathrm{L}^{p_0}$, we have
\begin{align*}
\sum_{l=-1}^{N+M_0} \norm{\Delta_l \nabla u_1}_\infty &\leq \norm{\Delta_{-1}\nabla u_1}_\infty + C \sum_{l=0}^{N+M_0} \norm{\Delta_l \omega_1}_\infty\\
&\leq C \norm{\Delta_{-1} \nabla u_1}_{p_0} + C \Pi(N+M_0)\norm{\omega_1}_\Pi\\
&\leq C \norm{\Delta_{-1}\omega_1}_{p_0} + C \Pi(N+M_0)\norm{\omega_1}_\Pi\\
&\leq C\norm{\omega_1}_{p_0} + C \Pi(N+M_0)\norm{\omega_1}_\Pi\\
&\leq C\norm{\omega_1}\Pi(N+M_0),
\end{align*}
where $\norm{\,\cdot\,} = \norm{\,\cdot\,}_{\Pi \cap \mathrm{L}^{p_0}}$ and we make use of the fact that $\Pi(\alpha) \geq 1$ for all $\alpha \in \Real$. Combined with the previous estimate, we conclude
\begin{equation}
A_1 \leq C \norm{\omega_1} \Pi(N+M_0)\sum_{l=-1}^{N+M_0} \norm{\Delta_l \rho}_\infty.
\end{equation}
For the second term, we write:
\begin{align*}
A_2 &= C \sum_{j=-1}^N \sum_{l \geq j- M_0}\sideset{}{'}\sum_{|l-m|\leq 1} 2^{j-l} \norm{\Delta_l \nabla u_1}_\infty \norm{\Delta_m \rho}_\infty\\
&\leq C \sum_{m=-1}^\infty \norm{\Delta_m \rho}_\infty \left(\sum_{j=-1}^{\min (N, m+ M_0 +1)} 2^{j-m}\right) \sideset{}{'}\sum_{|l-m|\leq 1} \norm{\Delta_l \nabla u_1}_\infty\\
&\leq C \sum_{m=-1}^\infty \left(2^{\min(N,m)-m} \sideset{}{'}\sum_{|l-m|\leq 1} \norm{\Delta_l \nabla u_1}_\infty\right)\norm{\Delta_m \rho}_\infty\\
&= \sum_{m=-1}^N \left( \sideset{}{'}\sum_{|l-m|\leq 1} \norm{\Delta_l \nabla u_1}_\infty\right)\norm{\Delta_m \rho}_\infty\\
&\hspace{30pt}+ \sum_{m=N+1}^\infty \left(2^{N-m} \sideset{}{'}\sum_{|l-m|\leq 1} \norm{\Delta_l \nabla u_1}_\infty\right)\norm{\Delta_m \rho}_\infty\\
&\leq C \Pi(N)\norm{\omega_1}\sum_{m=-1}^N \norm{\Delta_m \rho}_\infty + C \sum_{m=N+1}^\infty \left(2^{N-m} \Pi(m) \norm{\omega_1}\right)\norm{\Delta_m \rho}_\infty\\
&\leq C \Pi(N)\norm{\omega_1}\sum_{m=-1}^N \norm{\Delta_m \rho}_\infty\\
&\hspace{30pt} + C \sum_{m=N+1}^\infty \left(\int_N^\infty 2^{N-\xi}\Pi(\xi)\text{d}\xi \norm{\omega_1}\right)\norm{\Delta_m \rho}_\infty\\
&\leq C \Pi(N)\norm{\omega_1}\sum_{m=-1}^\infty \norm{\Delta_m \rho}_\infty. \hspace{15pt}(\text{by (iii)})
\end{align*}
Combined with the estimate for $A_1$, we conclude that
\begin{equation}\label{eq3.2}
\sum_{j=-1}^N \norm{R_j(u_1,\rho)}_\infty \leq C \norm{\omega_1}\Pi(N)\sum_{j=-1}^\infty \norm{\Delta_j \rho}_\infty.
\end{equation}

Next, we estimate $\norm{R_j(v, \rho_2)}_\infty$. Similar to $R_j(u_1,\rho)$, we split the estimate into two terms:
\begin{align*}
\norm{R_j(v,\rho_2)}_\infty \leq C \sum_{|j-l| \leq M_0} &\left\{ \norm{S_{l-2} v}_\infty \norm{\Delta_l \nabla \rho_2}_\infty +\norm{S_{l-2} \nabla \rho_2}_\infty \norm{\Delta_l v}_\infty \right\}\\
&+\; C 2^j \sideset{}{'}\sum_{\stackrel{l \geq j- M_0}{|l-m|\leq 1}} 2^{-l} \norm{\Delta_l \nabla v}_\infty \norm{\Delta_m \rho_2}_\infty\\
\end{align*}
and conclude that
\begin{equation}\label{eq3.3}
\sum_{j=-1}^N \norm{R_j(v,\rho_2)}_\infty \leq C \norm{\nabla \rho_2}\Pi(N)\sum_{j=-1}^\infty \norm{\Delta_j v}_\infty.
\end{equation}

Finally, we estimate $\norm{(S_{j-2}v,\nabla)\Delta_j \rho_2}_\infty$ and $\norm{(S_{j-2}u_1,\nabla)\Delta_j \rho}_\infty$. We have
\[\norm{(S_{j-2}v,\nabla)\Delta_j \rho_2}_\infty \leq \norm{S_{j-2}v}_\infty \norm{\Delta_j \nabla \rho_2}_\infty, \]
from which the estimate
\begin{align}\label{eq3.4}
\sum_{j=-1}^N \norm{(S_{j-2}v,\nabla)\Delta_j \rho_2}_\infty &\leq \left(\sup_{-1 \leq j \leq N} \norm{S_{j-2} v}_\infty\right) \sum_{j=-1}^N \norm{\Delta_j \nabla \rho_2}_\infty\\
&\nonumber\leq C \norm{\nabla \rho_2}\Pi(N)\sum_{j=-1}^N \norm{\Delta_j v}_\infty
\end{align}
easily follows. Similarly, we can write
\begin{equation}\label{eq3.4a}
\sum_{j=-1}^N \norm{(S_{j-2}u_1,\nabla)\Delta_j \rho}_\infty \leq C \norm{\omega_1}\Pi(N)\sum_{j=-1}^N \norm{\Delta_j \rho}_\infty.
\end{equation}
Combining (\ref{eq3.2}) and (\ref{eq3.3})-(\ref{eq3.4a}), we sum (\ref{eq3.1}) from $j=-1$ to $N$ and estimate it as:
\begin{equation}\label{eq3.5}
\sum_{j=-1}^N \norm{\Delta_j \rho (t)}_\infty \leq C\Pi(N)\int_0^t  \sum_{j=-1}^\infty \left(\norm{\Delta_j \rho(\tau)}_\infty + \norm{\Delta_j v(\tau)}_\infty\right)\text{d}\tau,
\end{equation}
where we have used (\ref{eq3.0.1}) to bound $\norm{\omega_1(\tau)}$, $\norm{\nabla \rho_2(\tau)}$ uniformly on $[0,T]$.

Next, we apply the $\Delta_j$ operator to $\dot{v}$ and find
\begin{align}\label{eq3.6}
\Delta_j \dot{v} +(S_{j-2} u_1,\nabla)\Delta_j v &= - R_j(u_1,v) - (S_{j-2}v,\nabla)u_2\\
& \nonumber \hspace{30pt} - R_j(v,u_2) - \Delta_j\nabla P + \left(
                                                         \begin{array}{c}
                                                           0 \\
                                                           \Delta_j\rho \\
                                                         \end{array}
                                                       \right).
\end{align}
Define the flow mappings $\{X_j(x, t;\tau)\}$ given by:
\begin{equation*}
\left\{\begin{array}{l}
\dot{X}_j(x, t;\tau) = S_{j-2} u_1(X_j(x, t;\tau),t)\\
X_j(x, 0;\tau) = x(\tau)
\end{array}\right.
\end{equation*}
(where $S_{j-2}$ is $S_{-1}$ when $j = -1, 0$). Integrating (\ref{eq3.6}) along $\{X_j(\alpha, t;\tau)\}$ and taking the $\mathrm{L}^\infty$-norm of both sides yields
\begin{align}\label{eq3.7}
\norm{\Delta_j v(t)}_\infty &\leq \int_0^t \norm{R_j(u_1,v)}_\infty + \norm{(S_{j-2}v,\nabla)u_2}_\infty\\
&\nonumber\hspace{20pt}+ \norm{R_j(v,u_2)}_\infty + \norm{\Delta_j\nabla P}_\infty + \norm{\Delta_j \rho}_\infty \text{d}\tau.
\end{align}
The estimates for the first three terms hew closely to estimates (\ref{eq3.2}), (\ref{eq3.4}) and (\ref{eq3.3}), respectively, and will therefore not be repeated. For the pressure term, we take the divergence of (\ref{eq3.6}) and use (\ref{eq3.0.3}) to find:
\begin{align}
\Delta_j \Delta P &= -\text{div} R_j(u_1, v) - \text{div}R_j(v, u_2) - \text{tr}(\nabla \Delta_j v \cdot \nabla S_{j-2}u_1)\\
&\nonumber\hspace{30pt} - \text{tr}(\nabla \Delta_j u_2 \cdot \nabla S_{j-2} v) + \Delta_j \partial_2 \rho.
\end{align}

We consider two cases, $j \geq 0$ and $j=-1$. For $j\geq 0$, observe first that
\[ -\nabla \Delta_j P(x) = \mathcal{F}_{\xi\rightarrow x} (i\xi|\xi|^{-2} (\Delta_j\Delta P)\;\hat{ }\; (\xi)).\]
Using the estimate for $\Delta_j \Delta P$ above, as well as a Littlewood-Paley argument and Bernstein's inequality gives:
\begin{align}
\norm{\Delta_j \nabla P}_\infty &\leq C\left(\norm{R_j(u_1, v)}_\infty + \norm{R_j(v,u_2)}_\infty\right.\\
&\nonumber \hspace{20pt} + 2^{-j}\norm{\nabla \Delta_j v}_\infty \norm{S_{j-2} \nabla u_1}_\infty\\
&\nonumber\hspace{20pt} + \left.2^{-j}\norm{\nabla \Delta_j u_2}_\infty \norm{S_{j-2} \nabla v}_\infty + \norm{\Delta_j \rho}_\infty\right)\\
&\nonumber \leq C\left(\norm{R_j(u_1, v)}_\infty + \norm{R_j(v,u_2)}_\infty + \norm{\Delta_j v}_\infty \norm{S_{j-2} \nabla u_1}_\infty\right.\\
&\nonumber \hspace{25pt}\left. + \norm{\Delta_j u_2}_\infty \norm{S_{j-2} \nabla v}_\infty + \norm{\Delta_j \rho}_\infty\right).
\end{align}
For $j=-1$, we use
\begin{align}
\norm{\Delta_{-1} \nabla P}_\infty &\leq C \norm{\Delta_{-1} \nabla P}_{p_2} \\
&\nonumber\leq C\norm{ - \sum_{k=1}^n \partial_k \Delta_{-1} \{u_1^{(k)} v + v^{(k)} u_2\}}_{p_2} + \norm{\Delta_{-1} \rho}_{p_2}\\
&\nonumber\leq C\norm{\Delta_{-1} (u_1 \otimes v)}_{p_2} + C\norm{\Delta_{-1} (v \otimes u_2)}_{p_2} +\norm{\Delta_{-1} \rho}_{p_2}\\
&\nonumber\leq C\sum_{|l-m|\leq M_0} \left(\norm{\Delta_l u_1}_{p_2} + \norm{\Delta_l u_2}_{p_2}\right)\norm{\Delta_m v}_\infty + \norm{\Delta_{-1} \rho}_{p_2},
\end{align}
where $p_2 \in [\frac{np_0}{n-p_0},\infty)$. Sobolev embedding and (\ref{eq3.0.1}) then gives the desired control over  pressure in terms of the other members of the right hand side of (\ref{eq3.7}).

Combining this estimate with those for the previous three terms and using (\ref{eq3.0.1}), we have:
\begin{equation}\label{eq3.8}
\sum_{j=-1}^N \norm{\Delta_j v(t)}_\infty \leq C\Pi(N)\int_0^t \sum_{j=-1}^\infty (\norm{\Delta_j v(\tau)}_\infty + \norm{\Delta_j \rho (\tau)}_\infty) \text{d}\tau,
\end{equation}
and therefore:
\begin{align}\label{eq3.9}
\sum_{j=-1}^N \norm{\Delta_j v(t)}_\infty + \norm{\Delta_j \rho(t)}_\infty &\leq\\
&\nonumber\hspace{-30pt} C\Pi(N)\int_0^t \sum_{j=-1}^\infty \norm{\Delta_j v(\tau)}_\infty + \norm{\Delta_j \rho (\tau)}_\infty \text{d}\tau.
\end{align}
We wish to use a Gronwall-type estimate, so we need to control\\ $\sum_{j=N+1}^\infty \left(\norm{\Delta_j v(t)}_\infty + \norm{\Delta_j \rho(t)}_\infty\right)$ as well. Control of these two quantities is nearly identical, so we address only $\rho(t)$. Suppressing time, we use Bernstein's inequality to write $\sum_{j=N+1}^\infty \norm{\Delta_j \rho}_\infty \leq C \sum_{j=N+1}^\infty 2^{-j}\norm{\Delta_j \nabla \rho}_\infty.$
Next, we set $d_k = \sum_{j=-1}^k \norm{\Delta_j \nabla \rho}_\infty$, and use Abel's lemma to write
\begin{align}\label{eq3.10}
\sum_{j=N+1}^\infty 2^{-j}\norm{\Delta_j \nabla \rho}_\infty &\leq \sum_{j=N+1}^\infty 2^{-j}(d_j - d_{j-1})\\
\nonumber &\leq -2^{(N-1)} d_N + \sum_{j=N+1}^\infty d_j (2^{-j}-2^{-(j+1)}) \\
\nonumber &\leq -2^{(N-1)} d_N + \norm{\nabla \rho}_\Pi\sum_{j=N+1}^\infty 2^{-j}\Pi(j)\\
\nonumber &\leq -2^{(N-1)} d_N + \norm{\nabla \rho}_\Pi\int_{j=N+1}^\infty 2^{-j}\Pi(j)\text{d}j\\
\nonumber &\leq C2^{-N}\Pi(N),
\end{align}
due to conditions (ii) and (iii) on $\Pi$ and (\ref{eq3.0.1}). This allows us to write:
\begin{equation}\label{eq3.11}
\sum_{j=-1}^\infty \left(\norm{\Delta_j v}_\infty + \norm{\Delta_j \rho}_\infty\right)\leq \sum_{j=-1}^N \left(\norm{\Delta_j v}_\infty + \norm{\Delta_j \rho}_\infty\right) + C2^{-N}\Pi(N).
\end{equation}

Set $F(t) = \int_0^t \sum_{j=-1}^\infty \left(\norm{\Delta_j v(\tau)}_\infty + \norm{\Delta_j \rho(\tau)}_\infty\right)\text{d}\tau,$
and use (\ref{eq3.9}) and (\ref{eq3.11}) to achieve the estimate
\[F'(t) \leq C\Pi(N)F(t) + C2^{-N}\Pi(N).\]
$F(t)$ is a monotonically nondecreasing, absolutely continuous function, and given (\ref{eq3.0.1}), we have $\norm{F'}_{\mathrm{L}^\infty([0,T])} \leq C.$ Since $F(0)=0$ by construction, this implies that there exists some $t_0$ such that
\[ F(t) \equiv 0 \text{ on } [0,t_0],\hspace{15pt} F(t) > 0 \text{ on } (t_0,T].\]
If $t_0 = T$, then we have uniqueness. Therefore we assume that $t_0 < T$. Fix $\eps > 0$ sufficiently small that $t_0 + \eps < T$ and $F(t) < 2^{-M_1 -1}$ on $(t_0, t_0 +\eps)$ (where $M_1$ will be determined later). Choose $t \in (t_0, t_0+\eps)$ and let $N = \max\{1,\lceil -\log_2 F(t) \rceil\}$. This gives
\begin{equation}\label{eq3.12}
F'(t) \leq C\Pi(-\log_2 F(t))F(t), \hspace{10pt} F(0) = 0.
\end{equation}
By (\ref{2assump1}), we have $\int_0^{1/2} F^{-1}(\Pi(-\log_2 F))^{-1}\text{d}F = C \int_1^\infty (\Pi(\alpha))^{-1}\text{d}\alpha = \infty,$ and the Osgood Uniqueness Theorem (see ~\cite{MR1688875} for statement) applies to
\begin{equation}\label{eq3.13}
\left\{\begin{array}{l}
\dot{\eta}(t,\delta) = C\Pi(-\log_2 \eta)\eta\vspace{3pt}\\
\eta(t_0,\delta) = \delta\\
\end{array}\right.
\end{equation}
for $\delta > 0$ sufficiently small. $\eta$ exists on $(t_0, t_0 + \eps)$ and depends continuously on $\delta$. By Gronwall's inequality, we have $F(t) < \eta(t,\delta)$ for all $t \in [t_0, t_0+\eps)$. To see this, suppose $t_1 = \min_{[t_0,t_0+\eps)}\{\;t\;| F(t) = \eta(t,\delta)\}$. Then we have
\begin{align*}
F(t_1) &\leq \int_{t_0}^{t_1}\Pi(-\log_2 F(\tau))F(\tau)\text{d}\tau\\
&\leq \int_{t_0}^{t_1}\Pi(-\log_2 \eta(\tau,\delta))\eta(\tau,\delta)\text{d}\tau\\
&< \delta + \int_{t_0}^{t_1}\Pi(-\log_2 \eta(\tau,\delta))\eta(\tau,\delta)\text{d}\tau\\
&= \eta(t_1,\delta),
\end{align*}
contradicting the definition of $t_1$. Finally, choose $M_1$  such that (\ref{2assump2}) holds for $\alpha \geq M_1$. As $\delta\rightarrow0^+$, we must have that $F \equiv 0$ on $[t_0, t_0 +\eps)$, contradicting the definition of $t_0$. This implies that $t_0 = T$, and uniqueness is proven.
\end{proof}

\section{Construction of the flow}\label{ExSec}
Let $\Gamma$, $\Gamma_1$ satisfy (i)-(vi). This section is dedicated to proving the following theorems:
\begin{thm}\label{existthm1}
For $1 < p_0 < 2 < p_1 < \infty$, let $f \in B_\Gamma \cap \mathrm{L}^{p_0} \cap \mathrm{L}^{p_1}$ and
$g \in \mathrm{W}^{1,p_0} \cap \mathrm{W}^{1,p_1}$ such that $\nabla g \in B_\Gamma$. Assume that $(\alpha +2)\Gamma'(\alpha) \leq C \text{ for a.e. } \alpha \in [-1,\infty)$. Then there exists a $T > 0$ (depending on $\Gamma$, $f$ and $g$) and a solution $(u,\rho)$ to the system of equations $(B_{\kappa, 0})$ with $u = \mathcal{K} * \omega$, such that
\begin{align}
\omega(\;\cdot\;) &\in \mathrm{L}^\infty ([0,T]; \mathrm{L}^{p_0} \cap \mathrm{L}^{p_1}) \cap C_{w^*}([0,T]; B_{\Gamma_1}),\\
\nabla \rho(\;\cdot\;) &\in \mathrm{L}^\infty ([0,T]; \mathrm{L}^{p_0} \cap \mathrm{L}^{p_1}) \cap C_{w^*}([0,T]; B_\Gamma).
\end{align}
\end{thm}

\begin{thm}\label{existthm2}
Let $f$ and $g$ be as in Theorem \ref{existthm1}. Assume that $\Gamma'(\alpha)\Gamma_1(\alpha)\leq C \text{ for a.e. }\\
 \alpha \in [-1,\infty)$. Then there exist $(u,\rho)$ solving $(B_{\kappa, 0})$ such that
\begin{align}
\omega(\;\cdot\;) &\in \mathrm{L}^\infty_{\text{loc}} ([0,\infty); \mathrm{L}^{p_0} \cap \mathrm{L}^{p_1}) \cap C_{w^*}([0,\infty); B_{\Gamma_1}),\\
\nabla \rho(\;\cdot\;) &\in \mathrm{L}^\infty_{\text{loc}} ([0,\infty); \mathrm{L}^{p_0} \cap \mathrm{L}^{p_1}) \cap C_{w^*}([0,\infty); B_\Gamma).
\end{align}
\end{thm}
\noindent It should be noted that $C_{w^*}([0,T];B_{\Gamma_1})$ is the space of weak-* continuous functions with values in $B_{\Gamma_1}$ in the sense of duality, $H_{\Gamma_1}' = B_{\Gamma_1}$. We define the predual $H_{\Gamma_1}$ of $B_{\Gamma_1}$ as:
\begin{align*}
H_{\Gamma_1}& =\left\{ f \in \mathcal{S}' \;\;\vline \;\;\exists \;\{d_j\}_{j=-1}^\infty,\; d_j \geq 0,\; \sum_{j=-1}^\infty d_j < \infty\right.\\
&\hspace{50pt}\left.\text{ and } \norm{\Delta_m f}_1 \leq \sum_{j\geq m} d_j \Gamma_1(j)^{-1} \;\forall \; m\geq -1 \right\}\end{align*}
and state without proof that the dual of $H_{\Gamma_1}$ is isomorphic to $B_{\Gamma_1}$ (For related discussion, see ~\cite{MR0461123}, Chapter 3).

\begin{proof}[Proof of Theorem \ref{existthm1}]
We use approximation by Sobolev-regular solutions to prove the existence theorem, for which we will need the following result from  ~\cite{MR2227730}:
\begin{prop}[Chae]\label{Chaethm}
Let $\kappa > 0$ be fixed, and div $u_0=0$. Let $r >2$ be an integer, and $(u_0, \rho_0) \in H^r(\Real^2)$.
Then there exists a unique solution $(u,\rho)$ to $(B_{\kappa, 0})$ with\\
$u \in C([0,\infty); H^r(\Real^2))$ and $\rho \in C([0,\infty); H^r(\Real^2)) \cap \mathrm{L}^2 ([0,T]; H^{r+1}(\Real^2))$.
\end{prop}
\noindent Following ~\cite{MR1851996}, we have that for $f\in \mathcal{S}'$, the $H^r$ Sobolev norm is equivalent to the following Littlewood-Paley decomposition:
\[
\norm{\Delta_{-1} f}_2 + \left(\sum_{k=0}^\infty 2^{2kr}\norm{\Delta_k f}_2^2\right)^{\frac{1}{2}}
\]
A simple application of the Cauchy-Schwarz inequality gives:
\begin{prop}\label{embed}
For $r > 1$, $H^r(\Real^2) \subset B_\Gamma(\Real^2).$
\end{prop}
For any $m\geq 1$, we construct the solution $(u_m,\rho_m)$ given by Proposition \ref{Chaethm} such that
\begin{equation}\label{eq4.1}
\omega_m(0) = S_m f \in \cap_{r>2} H^r\;\text{ and } \nabla \rho_m(0) = S_m \nabla g \in \cap_{r>2} H^r.
\end{equation}
Since $\norm{S_m h}_p \leq \norm{h}_p$ for $p\in [1,\infty]$ and any $h \in \mathcal{S}'$ by definition of the $S_m$ operator, we have
\begin{equation}\label{eq4.2}
 \norm{\omega_m (0)}_{p_0},\;\; \norm{\nabla \rho_m (0)}_{p_0}\leq C \;\text{and } \norm{\omega_m (0)}_{p_1},\;\; \norm{\nabla \rho_m (0)}_{p_1}\leq C.
\end{equation}
Furthermore, the definition of $\Delta_j$ and $S_m$ give us
\begin{equation}\label{eq4.3}
\norm{\omega_m(0)}_\Gamma \leq C\norm{f}_\Gamma, \hspace{15pt} \norm{\nabla \rho_m (0)}_\Gamma \leq C \norm{\nabla g}_\Gamma.
\end{equation}
Combined with Propositions \ref{Chaethm} and \ref{embed} we conclude that
\begin{align}\label{eq4.4}
&\omega_m(\cdot) \in \mathrm{L}^\infty_{loc} ([0,\infty); B_{\Gamma_1})\\
&\nonumber \nabla \rho_m(\cdot) \in \mathrm{L}^\infty_{loc} ([0,\infty); B_\Gamma).
\end{align}
Furthermore, using Theorem \ref{apriorithm} along with (\ref{eq4.2}), (\ref{eq4.3}) and (\ref{eq4.4}) we conclude that there is a $T > 0$ such that
\begin{align}\label{eq4.5}
&\omega_m(\cdot), \nabla \rho_m (\cdot) \in \mathrm{L}^\infty([0,T]; \mathrm{L}^{p_0} \cap \mathrm{L}^{p_1}),\\
&\nonumber \omega_m(\cdot) \in \mathrm{L}^\infty([0,T]; B_{\Gamma_1}),\\
&\nonumber \nabla \rho_m (\cdot) \in \mathrm{L}^\infty ([0,T];B_\Gamma).
\end{align}

Fix two indices, $m$ and $l$. Then we set
\begin{equation}\label{eq4.6}
\begin{array}{cc}
u_{\{m,l\}} = \mathcal{K} * \omega_{\{m,l\}}, & \omega = \omega_m - \omega_l\\
v = u_m - u_l, & \rho = \rho_m - \rho_l.\\
\end{array}
\end{equation}
We use the same estimate as in the uniqueness proof of Section \ref{USec}, only in this case we cannot assume that $\Delta_j v(0)$ and $\Delta_j  \rho (0)$ are zero. To wit, we have for any $N \geq 1$,
\begin{align}\label{eq4.7}
& \sum_{j=-1}^N \left(\norm{\Delta_j v(t)}_\infty + \norm{\Delta_j \rho(t)}_\infty\right)\\
&\nonumber \hspace{40pt}\leq \sum_{j=-1}^N \left(\norm{\Delta_j v(0)}_\infty + \norm{\Delta_j \rho(0)}_\infty\right)\\
&\nonumber \hspace{50pt}+ C\Gamma_1(N)\int_0^t \sum_{j=-1}^\infty (\norm{\Delta_j v(\tau)}_\infty + \norm{\Delta_j \rho (\tau)}_\infty) \text{d}\tau\\
&\nonumber \hspace{50pt} + C2^{-N}\Gamma_1(N).
\end{align}
As in the previous section, we define
\begin{equation*}
 F(t) = \int_0^t \sum_{j=-1}^\infty \left(\norm{\Delta_j v(\tau)}_\infty + \norm{\Delta_j \rho(\tau)}_\infty\right)\text{d}\tau,
\end{equation*}
which, for fixed $t \in [0,T]$ and $N = \max\{1,\lceil -\log_2 F(t) \rceil\}$ allows us to write (\ref{eq4.7}) as
\begin{equation}\label{eq4.8}
F'(t) \leq \sum_{j=-1}^N \left(\norm{\Delta_j v(0)}_\infty + \norm{\Delta_j \rho(0)}_\infty\right) + C \Gamma_1(-\log_2 F(t))F(t).
\end{equation}
Denote the first term on the right hand side by
\[\kappa_{m,l} + \iota_{m,l} := \sum_{j=-1}^N \norm{\Delta_j v(0)}_\infty + \sum_{j=-1}^N\norm{\Delta_j \rho(0)}_\infty. \]
The bounds on $\kappa_{m,l}$ and $\iota_{m,l}$ are similar so we demonstrate only the latter. To bound $\iota_{m,l}$, we first write
\begin{align*}
\iota_{m,l} &= \sum_{j=-1}^N \norm{\Delta_j \rho(0)}_\infty\\
\nonumber &\leq \sum_{j=-1}^N 2^{-j}\norm{\Delta_j (S_m-S_l)g}_\infty\\
\nonumber &= \sum_{j=-1}^N 2^{-j}\norm{\Delta_j (\sum_{k = l+1}^m \Delta_k g)}_\infty\\
\nonumber &\leq \sum_{k=l+1}^m \sum_{|k-j|\leq M_0} 2^{-k}\norm{\Delta_j \Delta_k g}_\infty\\
\nonumber &\leq C\sum_{k=l+1}^\infty  2^{-k}\norm{\Delta_k g}_\infty.
\end{align*}
Using an Abel's Lemma argument identical to (\ref{eq3.10}), we conclude that
\begin{equation}\label{eq4.9}
\iota_{m,l} \leq C 2^{-l} \Gamma(l).
\end{equation}
After integrating (\ref{eq4.8}) in time, we use (\ref{eq4.9}) to write
\[ F(t) \leq C 2^{-l}\Gamma(l) + C \int_0^t \Gamma_1(-\log_2 F(\tau))F(\tau)\text{d}\tau.\]
As in the proof of uniqueness, we have that $\Gamma_1(-\log_2 F)F$ is monotonically nondecreasing for small $F \geq 0$. If we let $\eta$ solve the ODE:
\[\left\{\begin{array}{l}
    \dot{\eta} = C \Gamma_1(-\log_2 \eta)\eta\\
    \eta(0) = C 2^{-l}\Gamma(l)\\
\end{array}\right.\]
Then a simple Gronwall argument gives $F(t) \leq \eta(t, C2^{-l}\Gamma(l))$ for $t\in [0,T]$. Combined with (\ref{eq4.8}), we have
\begin{equation}\label{eq4.10}
F'(t) \leq C 2^{-l}\Gamma(l) + C \Gamma_1[-\log_2 \eta(t, C2^{-l}\Gamma(l))]\eta(t, C2^{-l}\Gamma(l))
\end{equation}
for all $t\in [0,T]$. This implies that $\{u_m\}$ and $\{\rho_m\}$ are Cauchy sequences in the Banach space $\mathrm{L}^\infty([0,T];B_{\infty, 1}^0)$. Therefore, there exists $u, \rho$  such that:
\begin{equation}\label{eq4.11}
u_m \rightarrow u,\;\; \rho_m \rightarrow \rho \in \mathrm{L}^\infty([0,T];B_{\infty, 1}^0).
\end{equation}
As we show next, this in fact implies that for $ \omega = \text{curl } u$,
\begin{equation}\label{eq4.12}
\norm{\omega}_{\Gamma_1}, \norm{\nabla \rho}_\Gamma \in \mathrm{L}^\infty([0,T]).
\end{equation}
Define the seminorm $\nu_N$ on $\mathrm{L}^\infty([0,T];B_{\infty, 1}^0)$ given by
\begin{equation*}
\nu_N(f) = \norm{ \sum_{j=-1}^N \norm{\Delta_j f(\cdot)}_\infty}_{\mathrm{L}^\infty ([0,T])}.
\end{equation*}
By (\ref{eq4.11}), it is clear that $\nu_N(u_m - u)$ and $\nu_N(\rho_m-\rho)$ tend to zero as $m\rightarrow \infty$.
Using Bernstein's inequality, we have
\begin{equation*}
\norm{\sum_{j=-1}^N \norm{\Delta_j (\omega_m-\omega)}_\infty}_{\mathrm{L}^\infty ([0,T])} \leq C 2^N \nu_N(u_m-u),
\end{equation*}
and similarly for $\nabla \rho_m$. Working with $\nabla \rho$, this yields
\begin{align}\label{eq4.13}
&\norm{\sum_{j=-1}^N \norm{\Delta_j \nabla \rho_m}_\infty - \norm{\Delta_j \nabla \rho}_\infty}_{\mathrm{L}^\infty ([0,T])}\\ &\nonumber \hspace{30pt}\leq \norm{\sum_{j=-1}^N \;\vline\;\norm{\Delta_j \nabla \rho_m}_\infty - \norm{\Delta_j \nabla \rho}_\infty\;\vline\;}_{\mathrm{L}^\infty ([0,T])}\\
&\nonumber\hspace{30pt}\leq \norm{\sum_{j=-1}^N \norm{\Delta_j (\nabla \rho_m - \nabla \rho)}_\infty}_{\mathrm{L}^\infty ([0,T])}\\
 &\nonumber \hspace{30pt}\leq C 2^N \nu_N(\rho_m -\rho).
\end{align}
By (\ref{eq4.5}), we have that $\sum_{j=-1}^N \norm{\Delta_j \nabla \rho_m(t)}_\infty \leq C \Gamma(N),$
where $C$ is independent of our choice of $m$. Using (\ref{eq4.13}) and the above, we have
\begin{align*}
\norm{\sum_{j=-1}^N \norm{\Delta_j \nabla \rho}_\infty}_{\mathrm{L}^\infty ([0,T])} &\leq \norm{\sum_{j=-1}^N \norm{\Delta_j \nabla \rho}_\infty - \norm{\Delta_j \nabla \rho_m}_\infty}_{\mathrm{L}^\infty ([0,T])}\\
& \hspace{20pt} + \norm{\sum_{j=-1}^N \norm{\Delta_j \nabla \rho_m}_\infty}_{\mathrm{L}^\infty ([0,T])}\\
&\leq C 2^N \nu_N(\rho_m-\rho) + C \Gamma(N),
\end{align*}
and passing to the limit as $m \rightarrow \infty$ gives the second inclusion of (\ref{eq4.12}). The bound on $\omega$ is similar (with $\Gamma_1$ in place of $\Gamma$), and is therefore omitted.

It remains to show that $(u,\rho)$ satisfy the Boussinesq equations. Note that since $\{u_m\}$ and $\{\rho_m\}$ are in fact Cauchy sequences in $C([0,T]; B_{\infty, 1}^0)$ (in addition to $\mathrm{L}^\infty([0,T]; B_{\infty, 1}^0)$), we can conclude that
\begin{equation}\label{eq4.15}
u_m \rightarrow u, \rho_m \rightarrow \rho \text{ in } \mathrm{L}^\infty(\Real^2 \times [0,T]) \cap C(\Real^2 \times [0,T]),
\end{equation}
and after choosing a subsequence, we have:
\begin{align}
\label{eq4.16}&\omega_m \stackrel{w*}{\rightharpoonup} \omega,\;\; \nabla \rho_m \stackrel{w*}{\rightharpoonup} \nabla \rho \text{ in } \mathrm{L}^\infty([0,T];\mathrm{L}^{p_0}) \cap \mathrm{L}^\infty([0,T];\mathrm{L}^{p_1}),\\
\label{eq4.17}&\dot{u}_m \stackrel{w*}{\rightharpoonup} \dot{u},\;\; \dot{\rho}_m \stackrel{w*}{\rightharpoonup} \dot{\rho} \text{ in } \mathrm{L}^\infty([0,T];\mathrm{L}^{p_0}) \cap \mathrm{L}^\infty([0,T];\mathrm{L}^{p_1}).
\end{align}
Let $\beta \in \mathcal{S}$, $\text{div } \beta = 0$ be a test function, and let $\theta \in \mathcal{D}([0,T])$. By definition of $(u_m, \rho_m)$, we have
\begin{align*}
\langle u_m(0), \beta\rangle\theta(0) &+ \int_0^T \langle u_m(\tau),\beta\rangle\dot{\theta}(\tau)\\
& + \langle u_m(\tau), (u_m(\tau), \nabla)\beta\rangle\theta(\tau) - \langle \rho_m(\tau), \beta \rangle\theta(\tau)\text{d}\tau = 0,\\
\langle \rho_m(0), \beta\rangle\theta(0) &+ \int_0^T \langle\rho_m(\tau), \beta\rangle\dot{\theta}(\tau)\\
& + \langle\rho_m(\tau), (u_m(\tau),\nabla)\beta\rangle \theta(\tau) + \kappa \langle \rho_m(\tau), \Delta \beta\rangle \theta(\tau) \text{d}\tau = 0.
\end{align*}
From (\ref{eq4.1}) and the definition of $\omega_m$, $\rho_m$, we have
\begin{align*}
&\langle u_m(0), \beta\rangle \longrightarrow \langle \mathcal{K} * f,\beta\rangle\\
&\langle \rho_m(0), \beta \rangle \longrightarrow \langle g, \beta\rangle.
\end{align*}
Sending $m\rightarrow \infty$ and utilizing (\ref{eq4.15})-(\ref{eq4.17}), we conclude that $(u, \rho)$ solve $(B_{\kappa, 0})$.

It remains to show that $\omega(\cdot), \nabla \rho(\cdot)$ are weak-* continuous with values in $B_{\Gamma_1}$ and $B_\Gamma$, respectively. Since the proofs are nearly identical up to our choice of target space, we consider only the case of $\nabla \rho$. $\{\rho_m\}$ is a Cauchy sequence in $C([0,T];B^0_{\infty, 1})$, therefore we have that
\begin{equation*}
\norm{\nabla \rho- \nabla \rho_m}_{C([0,T];B^{-1}_{\infty, 1})} \rightarrow 0 \text{ as } m\rightarrow \infty.
\end{equation*}
Fix $h \in H_\Gamma$. Consider $\pi(t) := \langle \nabla \rho(t), h\rangle$ for $t\in [0,T]$, and define $\pi_m(t)$ similarly for $\rho_m(t)$. For any $t_0 \in [0,T]$, we have
\begin{equation}\label{eq4.18}
\pi(t) - \pi(t_0) = (\pi-\pi_m)(t) - (\pi-\pi_m)(t_0) +(\pi_m(t) - \pi_m(t_0)).
\end{equation}
By (\ref{eq4.4}), we have that for fixed $m$, $\pi_m(t) - \pi_m(t_0) \rightarrow 0$ as $t\rightarrow t_0$. For any $\tilde{h} \in H_\Gamma$, we have
\begin{equation*}
\abs{(\pi-\pi_m)(t)} \leq \abs{\langle(\nabla \rho - \nabla \rho_m)(t), h-\tilde{h}\rangle} + \abs{\langle(\nabla \rho - \nabla \rho_m)(t), \tilde{h}\rangle}.
\end{equation*}
Let $\norm{\cdot}_{\Gamma'}$ be the norm of $H_{\Gamma}$, given by $\norm{f}_{\Gamma'} = \inf_{\{d_j\}} \sum_{j=-1}^\infty d_j$. Next, for arbitrary $\delta > 0$, consider the space $B^1_{1,1 + \delta^{-1}}$. Note that any Besov Space based on the $\mathrm{L}^1$-norm contains every $\mathrm{L}^1$-function with bounded Fourier spectrum (since the Littlewood-Paley decomposition of such a function has only finitely many nonzero terms), and that these functions are dense in $H_\Gamma$. Therefore $B^1_{1,1 + \delta^{-1}}$ is dense in $H_\Gamma$, and we can choose $\tilde{h} \in B^1_{1, 1+ \delta^{-1}}$ such that $\norm{h-\tilde{h}}_{\Gamma'} < \eps$ for any $\eps > 0$. Because $\norm{\nabla \rho(t)}_\Gamma$ and $\norm{\nabla \rho_m(t)}_\Gamma$ are uniformly bounded on $[0,T]$, we have that
\begin{equation}
\abs{\langle (\nabla \rho - \nabla \rho_m)(t), h -\tilde{h}\rangle} < C\eps.
\end{equation}
Finally, we use the duality $(B^1_{1,1+ \delta^{-1}})' = B^{-1}_{\infty, 1 + \delta}$ and the embedding $B^{-1}_{\infty, 1} \hookrightarrow B^{-1}_{\infty, 1 + \delta}$, to write
\begin{equation}\label{eq4.19}
\abs{\langle (\nabla \rho -\nabla \rho_m)(t), \tilde{h}\rangle} \leq C\norm{\nabla \rho - \nabla \rho_m}_{C([0,T];B^{-1}_{\infty, 1})} \norm{\tilde{h}}_{B^1_{1,1 + \delta^{-1}}}.
\end{equation}
By choosing $m$ sufficiently large, we can make the right hand side of (\ref{eq4.19}) less than $\eps$. Combined with (\ref{eq4.17}), this gives $\limsup_{t\rightarrow t_0} |\pi(t)- \pi(t_0)| \leq C\eps$, which yields the desired result for $\nabla \rho$.
\end{proof}
\begin{proof}[Proof of Theorem \ref{existthm2}]
The proof follows that of Theorem \ref{existthm1}, except that our choice of $T > 0$ is arbitrary and no longer depends on $\Gamma$. Since the proof is identical to the above except in that respect, it is omitted.
\end{proof}

\nocite{MR639462} \nocite{MR808825} \nocite{MR2413097}
\bibliographystyle{amsalpha}
\bibliography{BBB}

\providecommand{\bysame}{\leavevmode\hbox to3em{\hrulefill}\thinspace}
\providecommand{\MR}{\relax\ifhmode\unskip\space\fi MR }
\providecommand{\MRhref}[2]{%
  \href{http://www.ams.org/mathscinet-getitem?mr=#1}{#2}
}
\providecommand{\href}[2]{#2}
\begin{thebibliography}{BKM84}

\bibitem[BC94]{MR1288809}
H.~Bahouri and J.-Y. Chemin, \emph{\'{E}quations de transport relatives \'a\
  des champs de vecteurs non-lipschitziens et m\'ecanique des fluides}, Arch.
  Rational Mech. Anal. \textbf{127} (1994), no.~2, 159--181. \MR{1288809
  (95g:35164)}

\bibitem[BKM84]{MR763762}
J.~T. Beale, T.~Kato, and A.~Majda, \emph{Remarks on the breakdown of smooth
  solutions for the {$3$}-{D} {E}uler equations}, Comm. Math. Phys. \textbf{94}
  (1984), no.~1, 61--66. \MR{763762 (85j:35154)}

\bibitem[CD80]{MR565993}
J.~R. Cannon and Emmanuele DiBenedetto, \emph{The initial value problem for the
  {B}oussinesq equations with data in {$L^{p}$}}, Approximation methods for
  {N}avier-{S}tokes problems ({P}roc. {S}ympos., {U}niv. {P}aderborn,
  {P}aderborn, 1979), Lecture Notes in Math., vol. 771, Springer, Berlin, 1980,
  pp.~129--144. \MR{565993 (81f:35101)}

\bibitem[Cha06]{MR2227730}
Dongho Chae, \emph{Global regularity for the 2{D} {B}oussinesq equations with
  partial viscosity terms}, Adv. Math. \textbf{203} (2006), no.~2, 497--513.
  \MR{2227730 (2007e:35223)}

\bibitem[Che98]{MR1688875}
Jean-Yves Chemin, \emph{Perfect incompressible fluids}, Oxford Lecture Series
  in Mathematics and its Applications, vol.~14, The Clarendon Press Oxford
  University Press, New York, 1998, Translated from the 1995 French original by
  Isabelle Gallagher and Dragos Iftimie. \MR{1688875 (2000a:76030)}

\bibitem[Che99]{MR1753481}
\bysame, \emph{Th\'eor\`emes d'unicit\'e pour le syst\`eme de {N}avier-{S}tokes
  tridimensionnel}, J. Anal. Math. \textbf{77} (1999), 27--50. \MR{1753481
  (2001c:35185)}

\bibitem[CKN99]{MR1711383}
Dongho Chae, Sung-Ki Kim, and Hee-Seok Nam, \emph{Local existence and blow-up
  criterion of {H}\"older continuous solutions of the {B}oussinesq equations},
  Nagoya Math. J. \textbf{155} (1999), 55--80. \MR{1711383 (2000g:35172)}

\bibitem[CN97]{MR1475638}
Dongho Chae and Hee-Seok Nam, \emph{Local existence and blow-up criterion for
  the {B}oussinesq equations}, Proc. Roy. Soc. Edinburgh Sect. A \textbf{127}
  (1997), no.~5, 935--946. \MR{1475638 (98e:35133)}

\bibitem[DP09]{MR2520505}
Rapha{\"e}l Danchin and Marius Paicu, \emph{Global well-posedness issues for
  the inviscid {B}oussinesq system with {Y}udovich's type data}, Comm. Math.
  Phys. \textbf{290} (2009), no.~1, 1--14. \MR{2520505 (2010f:35298)}

\bibitem[ES94]{MR1252833}
Weinan E and Chi-Wang Shu, \emph{Small-scale structures in {B}oussinesq
  convection}, Phys. Fluids \textbf{6} (1994), no.~1, 49--58. \MR{1252833
  (94i:76075)}

\bibitem[FJ85]{MR808825}
Michael Frazier and Bj{\"o}rn Jawerth, \emph{Decomposition of {B}esov spaces},
  Indiana Univ. Math. J. \textbf{34} (1985), no.~4, 777--799. \MR{808825
  (87h:46083)}

\bibitem[Guo89]{MR1013438}
Bo~Ling Guo, \emph{Spectral method for solving two-dimensional
  {N}ewton-{B}oussinesq equations}, Acta Math. Appl. Sinica (English Ser.)
  \textbf{5} (1989), no.~3, 208--218. \MR{1013438 (90i:35242)}

\bibitem[HK08]{MR2413097}
Taoufik Hmidi and Sahbi Keraani, \emph{Incompressible viscous flows in
  borderline {B}esov spaces}, Arch. Ration. Mech. Anal. \textbf{189} (2008),
  no.~2, 283--300. \MR{2413097 (2009j:35252)}

\bibitem[LWZ10]{MR2645152}
Xiaofeng Liu, Meng Wang, and Zhifei Zhang, \emph{Local well-posedness and
  blowup criterion of the {B}oussinesq equations in critical {B}esov spaces},
  J. Math. Fluid Mech. \textbf{12} (2010), no.~2, 280--292. \MR{2645152}

\bibitem[Maj03]{MR1965452}
Andrew Majda, \emph{Introduction to {PDE}s and waves for the atmosphere and
  ocean}, Courant Lecture Notes in Mathematics, vol.~9, New York University
  Courant Institute of Mathematical Sciences, New York, 2003. \MR{1965452
  (2004b:76152)}

\bibitem[MB02]{MR1867882}
Andrew~J. Majda and Andrea~L. Bertozzi, \emph{Vorticity and incompressible
  flow}, Cambridge Texts in Applied Mathematics, vol.~27, Cambridge University
  Press, Cambridge, 2002. \MR{1867882 (2003a:76002)}

\bibitem[Mey81]{MR639462}
Yves Meyer, \emph{Remarques sur un th\'eor\`eme de {J}.-{M}. {B}ony},
  Proceedings of the {S}eminar on {H}armonic {A}nalysis ({P}isa, 1980), no.
  suppl. 1, 1981, pp.~1--20. \MR{639462 (83b:35169)}

\bibitem[Pee76]{MR0461123}
Jaak Peetre, \emph{New thoughts on {B}esov spaces}, Mathematics Department,
  Duke University, Durham, N.C., 1976, Duke University Mathematics Series, No.
  1. \MR{0461123 (57 \#1108)}

\bibitem[PS92]{MR1167779}
Alain Pumir and Eric~D. Siggia, \emph{Development of singular solutions to the
  axisymmetric {E}uler equations}, Phys. Fluids A \textbf{4} (1992), no.~7,
  1472--1491. \MR{1167779 (93c:76014)}

\bibitem[Tan06]{MR2239910}
Yasushi Taniuchi, \emph{Remarks on global solvability of 2-{D} {B}oussinesq
  equations with non-decaying initial data}, Funkcial. Ekvac. \textbf{49}
  (2006), no.~1, 39--57. \MR{2239910 (2008d:35177)}

\bibitem[Tri01]{MR1851996}
Hans Triebel, \emph{The structure of functions}, Monographs in Mathematics,
  vol.~97, Birkh\"auser Verlag, Basel, 2001. \MR{1851996 (2002k:46087)}

\bibitem[Vis98]{MR1664597}
Misha Vishik, \emph{Hydrodynamics in {B}esov spaces}, Arch. Ration. Mech. Anal.
  \textbf{145} (1998), no.~3, 197--214. \MR{1664597 (2000a:35201)}

\bibitem[Vis99]{MR1717576}
\bysame, \emph{Incompressible flows of an ideal fluid with vorticity in
  borderline spaces of {B}esov type}, 1999, pp.~769--812. \MR{1717576
  (2000i:76008)}

\bibitem[Yud63]{MR0158189}
V.~I. Yudovich, \emph{Non-stationary flows of an ideal incompressible fluid},
  \u Z. Vy\v cisl. Mat. i Mat. Fiz. \textbf{3} (1963), 1032--1066. \MR{0158189
  (28 \#1415)}

\end{thebibliography}
\end{document}